\documentclass[11pt]{amsart}
\usepackage{graphicx}
\usepackage{amsmath, amssymb}
\usepackage{verbatim}
\usepackage{color}
\newtheorem{thm}{Theorem}[section]

\newtheorem{conj.}[thm]{Conjecture}
\newtheorem{lem}[thm]{Lemma}
\newtheorem{prop}[thm]{Proposition}
\theoremstyle{definition}
\newtheorem{defn}[thm]{Definition}
\theoremstyle{remark}

\numberwithin{equation}{section}

\newcommand{\h}{\mathcal{H}}

\begin{document}

\title[On Controlled Frames in Hilbert $C^*$-modules]
{On Controlled Frames  in Hilbert $C^*$-modules }%
\author[M. Rashidi-Kouchi, A. Rahimi]{M. Rashidi-Kouchi, A. Rahimi}

\address{ Department of Mathematics, Kahnooj Branch, Islamic Azad University, Kerman, Iran.}
\email{m$_{-}$rashidi@kahnoojiau.ac.ir}
\address{ Department of Mathematics, University of Maragheh, P. O. Box 55181-83111,
Maragheh, Iran.}

\email{rahimi@maragheh.ac.ir}


\begin{abstract}
In this paper, we introduce  controlled frames in Hilbert
$C^*$-modules and we show that they share many useful properties
with their corresponding notions in Hilbert space. Next, we give a
characterization of controlled frames in Hilbert $C^*$-module. Also
multiplier operators for controlled frames in Hilbert $C^*$-modules
will be defined and some of its properties will be shown. Finally,
we investigate weighted frames in Hilbert $C^*$-modules and verify
their relations to controlled frames and multiplier operators.

\end{abstract}
\maketitle
\section{Introduction}

Frames for Hilbert spaces were first introduced in 1952 by Duffin
and Schaeffer \cite{DS} for study of nonharmonic Fourier series.
They were reintroduced and development in 1986 by Daubechies,
Grossmann and Meyer\cite{DGM}, and popularized from then on.  For
basic results on frames, see \cite{C}.

Hilbert $C^*$-modules form a wide category between Hilbert spaces
and Banach spaces. Their structure was first used by Kaplansky
\cite{kap} in 1952. They are an often used tool in operator theory
and in operator algebra theory. They serve as a major class of
examples in operator $C^*$-module theory.

The notions of frames in Hilbert $C^*$-modules were introduced and
investigated in \cite{FL1}. Frank and Larson \cite{FL1,FL} defined
the standard frames in Hilbert $C^*$-modules in 1998 and got a
series of result for standard frames in finitely or countably
generated Hilbert $C^*$-modules over unital $C^*$-algebras.
Extending the results to this more general framework is not a
routine generalization, as there are essential differences between
Hilbert $C^*$-modules and Hilbert spaces. For example, any closed
subspace in a Hilbert space has an orthogonal complement, but this
fails in Hilbert $C^*$-module. Also there is no explicit analogue of
the Riesz representation theorem of continuous functionals in
Hilbert $C^*$-modules. We refer the readers to \cite{L} and
\cite{kasp} for more details on Hilbert $C^*$-modules and to
\cite{FL, RN1, RNM, RN} for a discussion of basic properties of
frame in Hilbert $C^*$-modules and their generalizations.

Controlled frames have been introduced  to improve the numerical
efficiency of iterative algorithms for inverting the frame operator
on abstract Hilbert spaces \cite{BAG}, however they are used earlier
in \cite{BVA} for spherical wavelets.

In this paper, we introduce  controlled frames in Hilbert
$C^*$-modules and we show that they share many useful properties
with their corresponding notions in Hilbert space. Next, we give a
characterization of controlled frames in Hilbert $C^*$-module. Also
multiplier operators for controlled frames in Hilbert $C^*$-modules
will be defined and some of its properties will be shown. Finally,
we investigate weighted frames in Hilbert $C^*$-modules and verify
their relation to controlled frames and multiplier operators.

 The paper is organized as follows. In section 2, we review the concept  Hilbert
$C^*$-modules, frames and multiplier operators in Hilbert
$C^*$-modules.  Also the analysis, synthesis, frame operator and
dual frames be reviewed. In section 3, we introduce  controlled
frames   in Hilbert $C^*$-modules and characterize them. In section
4, we investigate weighted frames in Hilbert $C^*$-modules and
verify their relation to controlled frames and multiplier operators.

\section{Preliminaries}

In this section, we collect the basic notations and some preliminary
results.  We denote by $I$ the identity operator on $\h$. Let
$B(\h_1,\h_2)$ be the set of all bounded linear operators from
$\h_1$ to $\h_2$. This set is a Banach space for the operator norm
$\|A\|=sup_{\|x\|_{\h_1}\leq 1 }\|Ax\|_{\h_2}$. The adjoint of the
operator $A$ is denoted by $A^*$ and the spectrum of $A$ by
$\sigma(A)$. We define $GL(\h_1,\h_2)$ as the set of all bounded
linear operators with a bounded inverse, and similarly for $GL(\h)$.
Our standard reference for Hilbert space and operator theory is
\cite{Conway}.

\par Controlled frames  introduced in \cite{BAG} as follows.
 Let $C \in GL(\h)$. A frame controlled by the operator $C$ or
$C$-controlled frame is a family of vectors $\Psi = \{\psi_j \in \h
: j \in J\}$, such that there exist two constants $m > 0$ and $M <
\infty$ satisfying \[ m\|f\|^2 \leq \sum_j \langle f, \psi_j \rangle
\langle C\psi_j , f\rangle \leq M\|f\|^2,\] for all $f \in \h$.

Also weighted frames are defined as follows. Let $\Psi = \{\psi_j
\in \h : j \in J\}$ be a sequence of elements in $\h$ and
$\{\omega_j : j\in J\} \subseteq \mathbb{R}^+ $a sequence of
positive weights. This pair is called a weighted frame of $\h$ if
there exist constants $m
> 0$ and $M < \infty$ such that
\[ m\|f\|^2 \leq \sum_j \omega_j |\langle f,
\psi_j \rangle |^2 \leq M\|f\|^2,\] for all $f \in \h$.

Hilbert $C^*$-modules form a wide category between Hilbert spaces
and Banach spaces. Hilbert $C^*$-modules are generalizations of
Hilbert spaces by allowing the inner product to take values in a
$C^*$-algebra rather than in the field of complex numbers.

Let $A$ be a $C^*$-algebra with involution $*$. An inner product
$A$-module (or pre Hilbert $A$-module) is a complex linear space
$\h$ which is a left $A$-module with an inner product map $ \langle.
,.\rangle:\h\times{\h} \rightarrow A $ which satisfies the following
properties:
\begin{enumerate}
\item $\langle \alpha{f}+\beta{g},h\rangle= \alpha\langle
f,h\rangle+\beta\langle g,h\rangle$ for all $f,g,h \in \h$ and
$\alpha, \beta \in \mathbb C$;
\item $\langle af,g\rangle =a\langle f,g\rangle$ for  all $f,g \in \h$ and $a\in A$;
\item $\langle f,g\rangle = \langle g,f\rangle^{*} $ for all $ f,g \in \h $;
\item $\langle f,f\rangle \geq 0$ for all $f\in \h$ and $\langle f,f\rangle=0$ iff $f=0$.
\end{enumerate}
For $f\in \h$, we define a norm on $\h$ by
$\|f\|_{\h}=\|\langle{f},{f}\rangle\|^{{1}/{2}}_{A}$. If $\h$ is
complete with this norm, it is called a (left) Hilbert $C^*$-module
over $A$ or a (left) Hilbert $A$-module.

An element $a$ of a $C^*$-algebra $A$ is positive if $a^* = a$ and
its spectrum  is a subset of positive real numbers. In this case, we
write $a\geq 0$. By condition (4) in the definition
$\langle{f},{f}\rangle \geq 0$ for every $f \in \h$, hence we define
$|f|=\langle{f},{f}\rangle ^{1/2}$. We call $Z(A)=\{a\in A: ab=ba, \forall b\in A\}$, the center of $A$. If $a\in Z(A)$, then $a^* \in Z(A)$, and if $a$ is an invertible element of $Z(A)$, then $a^{-1}\in Z(A)$, also if $a$ is a positive element of $Z(A)$, then $a^{\frac{1}{2}}\in Z(A)$. Let $Hom_A(M,N)$ denotes the set
of all $A$-linear operators from $M$ to $N$.

 Let
\[\ell^{2}(A)= \left\{ \{a_{j}\}\subseteq A:\sum_{j\in J}a_{j}^{*}a_{j}\,\, converges\,\,\, in \|. \|
\right\}\] with inner product
\[
\langle\{a_j\},\{b_j\}\rangle=\sum_{j\in J} a^*_j b_j,\quad
\{a_{j}\}, \{b_{j}\}\in\ell^{2}(A)
\]
and
\[
\|\{a_j\}\|:=\sqrt{\|\sum a^*_ja_j\|},
\] it was shown that \cite{O}, $\ell^{2}(A)$ is Hilbert $A$-module.
 \par Note that in Hilbert $C^*$-modules the
Cauchy-Schwartz inequality is valid.

Let $f,g\in\h$, where $\h$ is a Hilbert $C^*$-module , then
\[
\|\langle f,g\rangle\|^2\leq \|\langle f,f\rangle\|\times\|\langle
g,g\rangle\|.
\]

We are focusing in finitely and countably generated Hilbert $C^*$-
modules over unital $C^*$-algebra $A$. A Hilbert $A$-module $\h$ is
finitely generated if there exists a finite set $\{ x_1, x_2,
...,x_n\}\subseteq\h$ such that every $x\in\h$ can be expressed as
$x=\sum_{i=1}^{n}a_ix_i$, $a_i\in A$. A Hilbert $A$-module $\h$ is
countably generated if there exits a countable set of generators.

The notion of (standard) frames in Hilbert $C^*$-modules is first
defined by Frank and Larson \cite{FL}. Basic properties of frames in
Hilbert $C^*$-modules are discussed in  \cite{HJLM, HJM, KK1, KK2}.

Let $\h$ be a Hilbert $C^*$-module, and $J$ a set which is finite or
countable. A system $\{f_{j}: j\in J\}\subseteq \h$ is called a
frame for $\h$ if there exist constants $C, D > 0$ such that \\
\begin{equation}\label{101}
 C\langle{f},{f}\rangle\leq \sum_{j \in J}\langle{f},{f_{j}}\rangle\langle{f_{j}},{f}\rangle\leq D\langle{f},{f}\rangle
\end{equation}
 for all $f \in \h$. The constants $C$ and $D$ are called the frame
 bounds. If $C=D$ it called a tight frame and in the case $C=D=1$, it called Parseval frame. It is called a Bessel sequence if the second inequality in (\ref{101}) holds.

Unlike Banach spaces, it is known \cite{FL} that every finitely
generated or countably generated Hilbert $C^*$- modules admits a
frame.

The following characterization of frames in Hilbert $C^*$- modules,
which was obtained independently in \cite{A} and \cite{jing},
enables us to verify whether a sequence is a frames in Hilbert
$C^*$- modules in terms of norms. It also allows us to characterize
frames in Hilbert $C^*$- modules from the operator theory point of
view.
\begin{thm}\label{equivalent frame}
Let $\h$ be a finitely or countably generated Hilbert $A$-module
over a unital $C^*$-algebra $A$ and $\{f_{j}: j\in J\}\subseteq \h$
a sequence. Then $\{f_{j}: j\in J\}$ is a frame for $\h$  if and
only if there exist constants $C, D > 0$ such that
\[C\|f\|^{2}\leq \left\|\sum_{j\in J}\langle f, f_{j}\rangle  \langle f_{j}, f\rangle \right\|\leq D\|f\|^{2} ,\;\;\; f\in\h.\]
\end{thm}

Let $\{f_{j}: j\in J\}$ be a frame in Hilbert $A$-module $\h$ and
$\{g_{j}: j\in J\}$ be a sequence of $\h$. Then $\{g_{j}: j\in J\}$
is called a dual sequence of $\{f_{j}: j\in J\}$ if
\[ f= \sum_{j\in J}\langle f,g_{j}\rangle f_{j}\] for all $f\in \h$.
The sequences $\{f_{j}: j\in J\}$ and $\{g_{j}: j\in J\}$ are called
a dual frame pair when $\{g_{j}: j\in J\}$ is also a frame.

 For the frame $\{f_{j}: j\in J\}$  in Hilbert $A$-module $\h$, the operator $S$ defined by
\[ Sf=\sum_{j\in J}\langle f,f_{j}\rangle f_{j},\,\,\,  f\in \h\] is called the frame operator. It is proved that \cite{FL}, $S$ is
invertible, positive, adjointable and self-adjoint. Since
\[ \langle Sf, f\rangle = \langle \sum_{j\in J}\langle f,f_{j}\rangle f_{j}, f\rangle =  \sum_{j\in J}\langle f,f_{j}\rangle \langle f_{j}, f\rangle, \quad f\in\h\] it follows that
\[C\langle f, f\rangle \leq \langle Sf, f\rangle \leq D\langle f, f\rangle,\quad f\in\h \]
and the following reconstruction formula holds
\[ f= SS^{-1}f = S^{-1}Sf = \sum_{j\in J}\langle S^{-1}f,f_{j}\rangle f_{j}=\sum_{j\in J}\langle f,S^{-1}f_{j}\rangle f_{j} \] for all $f\in \h$.\\
 Let
$\tilde f_{j}= S^{-1}f_{j}$, then
\[f= \sum_{j\in J}\langle f,\tilde{f_{j}}\rangle f_{j}= \sum_{j\in J}\langle f,f_{j}\rangle \tilde{f_{j}},\] for any $f\in\h$.
The sequence $\{\tilde f_{j}: j\in J\} $ is also a frame for $\h$
which is called the canonical dual frame of $\{f_{j}: j\in J\}$.\\

 In \cite{schatt1}, R. Schatten provided a detailed
study of ideals of compact operators using their singular
decomposition. He investigated the operators of the form
$\sum_{j}\lambda_{j}\varphi_{j}\otimes\psi_{j}$ where $(\phi_{j})$
and $(\psi_{j})$ are orthonormal families. In \cite{xxlmult1}, the
orthonormal families were replaced with Bessel and frame sequences
to define Bessel and frame multipliers.
\begin{defn}
Let $\mathcal{\h}_{1}$ and $\mathcal{\h}_{2}$ be Hilbert spaces, let
$(\psi_{j})\subseteq\mathcal{\h}_{1}$ and
$(\phi_{j})\subseteq\mathcal{H}_{2}$ be Bessel sequences. Fix
$m=(m_j)\in l^{\infty}$. The operator $${\bf M}_{m, ( \phi_j),
(\psi_j)} : \mathcal{\h}_{1} \rightarrow \mathcal{\h}_{2}$$ defined
by
$$ {\bf M}_{m, (\phi_j), ( \psi_j )} (f)  =  \sum \limits_{j\in J} m_j
\langle f,\psi_j\rangle \phi_j $$ for all $f\in\h_1$ is called the
Bessel multiplier for the Bessel sequences $(\psi_{j})$ and
$(\phi_{j})$. The sequence $m$ is called the symbol of {\bf M}. For
frames, this operator is called frame multiplier, for Riesz
sequences a Riesz multiplier.
\end{defn}
Basic properties and some applications of this operator for Bessel
sequences,
 frames and Riesz basis have been proved by Peter Balazs in his Ph.D habilation \cite{peterthesis}.
 Recently, the concept of multipliers extended and introduced for continuous frames \cite{raba}, fusion frames \cite{aris}, $p$-Bessel
  sequences \cite{AsPe}, generalized frames \cite{5}, controlled frames \cite{rafe}, Banach frames \cite{FaOsRa,RAS}, Hilbert $C^*$-modules \cite{KM} and etc.

\begin{defn}\label{mul}
Let $A$ be a unital $C^*$-algebra, $J$ be a finite or countable
index set and $\{f_{j}: j\in J\}$ and $\{g_{j}: j\in J\}$ be Hilbert
$C^*$-modules Bessel sequences for $\h$. For $m=\{m_j\}_{j\in J}\in\ell^\infty(A)$ with $m_j \in Z(A)$, for each $j \in J$, 
the operator $M_{m,\{f_j\},\{g_j\}}:\h\to\h $ defined by
$$
M_{m,\{f_j\},\{g_j\}} f:=\sum_{j\in J}m_j\langle f,f_j\rangle
g_j,\quad f\in\h
$$

called the multiplier operator of $\{f_{j}: j\in J\}$ and $\{g_{j}:
j\in J\}$. The sequence $m=\{m_j\}_j\in J$ called the symbol of
$M_{m,\{f_j\},\{g_j\}}$.
\end{defn}
The symbol of  $m$ has important role in the studying of multiplier
operators. In this paper $m$ is always a sequence $m=\{m_j\}_{j\in J}\in\ell^\infty(A)$ with $m_j \in Z(A)$, for each $j \in J$.

\section{Controlled Frames In Hilbert $C^*$-modules}

In this section,  we introduce  controlled frames in Hilbert
$C^*$-modules and we show that they share many useful properties
with their corresponding notions in Hilbert space. We also give a
characterization of controlled frames in Hilbert $C^*$-module.

\begin{defn}
Let $\h$ be a Hilbert $C^*$-module and $C \in GL(\h)$. A frame
controlled by the operator $C$ or $C$-controlled frame in Hilbert
C*-module $\h$ is a family of vectors $\Psi = \{\psi_j \in \h : j
\in J\}$, such that there exist two constants $m > 0$ and $M <
\infty$ satisfying \[ m\langle f,f \rangle \leq \sum_{j\in J}
\langle f, \psi_j \rangle \langle C\psi_j , f\rangle \leq M\langle
f,f \rangle,\]  for all $f \in \h$.

Likewise, $\Psi = \{\psi_j \in \h : j \in J\}$ is called a $C$-controlled Bessel sequence with bound $M$ if there exists $M < \infty$ such that
\[  \sum_{j\in J} \langle f, \psi_j \rangle \langle C\psi_j , f\rangle \leq M\langle f,f \rangle,\]
for every $f \in \h$, where the sum in the inequality is convergent
in norm.

If $m=M$, we call this $C$-controlled frame a tight $C$-controlled
frame, and if $m=M=1$ it is called a Parseval $C$-controlled frame.

\end{defn}
Every frame is a $I$-controlled frame.  Hence controlled frames are
generalizations of frames.

The proof of the following lemma is straightforward.
\begin{lem}\label{controlled Bessel}
Let $\h$ be a Hilbert $C^*$-module and $C \in GL(\h)$. A sequence
$\Psi = \{\psi_j \in \h : j \in J\}$ is $C$-controlled Bessel
sequence in Hilbert C*-module $\h$ if and only if the operator \[S_C
f =\sum_{j\in J} \langle f, \psi_j \rangle C\psi_j\] is well defined
and there exists constant $M < \infty$ such that
 \[  \sum_{j\in J} \langle f, \psi_j \rangle \langle C\psi_j , f\rangle \leq M\langle f,f \rangle,\]
 for every $f \in \h$.
\end{lem}

By using  Lemma \ref{controlled Bessel} we have the following
definition.

\begin{defn}
  Let $\h$ be a Hilbert $C^*$-module and $C \in GL(\h)$. Assume the sequence $\Psi
= \{\psi_j \in \h : j \in J\}$ is $C$-controlled Bessel sequence in
Hilbert C*-module $\h$. The operator
\[S_C f =\sum_{j\in J} \langle f, \psi_j \rangle C\psi_j\]
is  called $C$-controlled frame operator.
\end{defn}

According to the following proposition, the main properties of
controlled frame operators in Hilbert $C^*$-modules are the same as
controlled frame operators in Hilbert spaces.

\begin{prop}\label{frame operator}
Let $\h$ be a Hilbert $C^*$-module on $C^*$-algebra $A$ and $C \in
End_A^*(U,V)$. Assume $\{\psi_j : j \in J\}$ is a C-controlled frame
in Hilbert $C^*$-module $\h$ with bounds $m,M>0$ Then $C$-controlled
frame operator $S_{C}$ is invertible, positive, adjointable and
self-adjoint.
\end{prop}

\begin{proof}
Let $f,g \in \h$ and $ Sf= \sum_{j\in J} \langle f, \psi_j\rangle
\psi_j $ be frame operator of  $\{\psi_j : j \in J\}$, then
\begin{align*}
    \langle S_Cf,g \rangle  & =\left\langle \sum_{j\in J} \langle
f, \psi_j\rangle C\psi_j, g \right\rangle= \sum_{j\in J} \langle f,
\psi_j\rangle \langle C\psi_j , g\rangle \\ & = \sum_{j\in J}
\langle f, \psi_j\rangle \langle \psi_j , C^* g\rangle = \langle Sf,
C^* g\rangle = \langle f, SC^* g\rangle.
\end{align*}
Therefore controlled frame operator $S_C$ is adjointable and $S_C^*=
SC^*$.

Since \[\langle S_Cf,f \rangle =\left\langle \sum_{j\in J} \langle
f, \psi_j\rangle C\psi_j, f \right\rangle= \sum_{j\in J} \langle f,
\psi_j\rangle \langle C\psi_j , f\rangle. \] It follows that
\[m\langle f,f \rangle \leq \langle S_C f,f\rangle \leq M\langle f,f\rangle,\]
and \[mId_{\h}\leq S_C \leq MId_{\h}.\] So $S_C$ is positive and
invertible. By regard to this fact that every positive operator in
Banach space is self-adjoint, $C$-controlled frame operator $S_{C}$
is self-adjoint.
\end{proof}

Now, by using the following lemma in \cite{A}, we give a
characterization of controlled frames.

\begin{lem}\cite{A}\label{A}
 Let $A$ be a $C^*$-algebra, $U$ and $V$ two Hilbert $A$-modules, and $T \in End_A^*(U,V)$. Then the following statements are equivalent:
\begin{enumerate}
         \item T is surjective;
         \item  $T^*$ is bounded below with respect to norm, that is, there is $m > 0$ such that
          $\|T^*f\|\geq m\|f\|$ for all $f \in U$;
         \item $T^*$ is bounded below with respect to the inner product, that is, there is $m'> 0$ such that
$\langle T^*f, T^*f\rangle \geq m'\langle f, f\rangle$ for all $f
\in U$.
       \end{enumerate}
\end{lem}

\begin{thm}
Let $\h$ be a Hilbert $C^*$-module and $C \in GL(\h)$. A sequence
$\Psi = \{\psi_j \in \h : j \in J\}$ is $C$-controlled frame in
Hilbert $C^*$-module $\h$ if and only if there exists constants $m >
0$ and $M < \infty$ such that

\begin{equation}\label{ineq}
 m\| f\|^2 \leq \left\| \sum_{j\in J} \langle f, \psi_j \rangle
\langle C\psi_j , f\rangle \right\| \leq M\|f \|^2,
\end{equation}

 for all $f \in \h$.
\end{thm}
\begin{proof}
Let the sequence $\Psi = \{\psi_j \in \h : j \in J\}$ is
$C$-controlled frame in Hilbert $C^*$-module $\h$. By the definition
of $C$-controlled frame inequality (\ref{ineq}) holds.

Now suppose that the inequality (\ref{ineq}) holds. From Proposition
\ref{frame operator} $C$-controlled frame operator $S_C$ is
positive, self-adjoint and invertible, hence $\langle
S_C^\frac{1}{2}f, S_C^\frac{1}{2}f\rangle = \langle S_Cf, f\rangle =
\sum_{j\in J} \langle f, \psi_j \rangle \langle C\psi_j , f\rangle
$. So we have $\sqrt{m}\|f\|\leq \|S_C^\frac{1}{2}f\|\leq
\sqrt{M}\|f\|$ for any $f\in \h$. According  to Lemma \ref{A} there
are constants $C, D>0$ such that
\[ C\| f\|^2 \leq \left\| \sum_{j\in J} \langle f,
\psi_j \rangle \langle C\psi_j , f\rangle \right\| \leq D\|f \|^2,\]
which implies that $\Psi = \{\psi_j \in \h : j \in J\}$ is
$C$-controlled frame in Hilbert $C^*$-module $\h$.
\end{proof}

 We prove the following theorems to show that every controlled frame
 is a classical frame.

\begin{thm}\label{efo}
Let $\h$ be a Hilbert $C^*$-module on $C^*$-algebra $A$  and
$\{\psi_j\}_{j\in J}$ be a sequence in $\h$. Assume $S$ is the frame
operator associated with $\{\psi_j\}_{j\in J}$ i.e. $Sf=\sum_{j\in
J} \langle f, \psi_j\rangle \psi_j$. Then the following conditions
are equivalent:
\begin{enumerate}
    \item $\{\psi_j\}_{j\in J}$ is a frame with frame bounds $C$ and
    $D$;
    \item  We have $CId_{\h}\leq S \leq DId_{\h}$.
\end{enumerate}

\end{thm}
\begin{proof}
$(1)\Rightarrow (2).$ This is implied by proof of Proposition
\ref{frame operator}.

$(2)\Rightarrow (1).$ Let $T$ be the analysis operator associated
with $\{\psi_j\}_{j\in J}$. Since $S=T^*T $  and hence
$\|T\|^2=\|S\|$ for any $f\in \h$ we obtain
\[\|\sum_{j\in J}\langle f, \psi_j\rangle \langle \psi_j, f\rangle\|=\|\langle Sf, f\rangle \|\leq \|Sf\|\|f\|^2\leq D\|f\|^2.\]
Also, for all $f\in \h$,
\[\|\langle Sf, f\rangle \| \geq \|\langle Cf, f\rangle\| = C\|f\|^2. \]
Therefor for all $f\in \h$
\[C\|f\|^2 \leq \|\sum_{j\in J}\langle f, \psi_j\rangle \langle \psi_j, f\rangle\| \leq D\|f\|^2. \]
Now, by Theorem \ref{equivalent frame} proof is complete.
\end{proof}

\begin{thm}\label{invertible}\cite{C}
Let $X$ be a Banach space, $U:X\rightarrow X$ a bounded operator and
$\|I-U\|< 1$. Then $U$ is invertible.
\end{thm}

\begin{thm}\label{Balaz2010}
Let $T:\h\rightarrow \h$ be a linear operator. Then the following
conditions are equivalent:
\begin{enumerate}
    \item There exist $m>0$ and $M< \infty$ such that $mI\leq T\leq
    MI$;
    \item $T\in GL^+(\h)$ i.e.$T$ is invertible and positive.
\end{enumerate}
\end{thm}
\begin{proof}
$(1)\Rightarrow (2)$ Since $mI\leq T$, then $0<(T-mI)$, therefore
\[0< \langle (T-mI)f,f\rangle= \langle Tf,f\rangle -\langle mf,
f\rangle\]. Hence
\[0< m\langle f,f\rangle \leq \langle Tf,f\rangle,\]
therefore $T$ is positive.

Since $mI\leq T\leq MI$, then $0\leq I-M^{-1}T\leq
\left(\frac{M-m}{M}\right)$. Therefore $\|I-M^{-1}T\|\leq
\left(\frac{M-m}{M}\right)< 1$. Now by Theorem \ref{invertible} the
operator $T$ is invertible.

$(1)\Rightarrow (2)$ Let $\|T\|=M$. Since $\|Th\|\leq \|T\|\|h\|$
for every $h\in \h$, then $\|Th\|\leq M \|h\|$ for every $h\in \h$,
therefore $T<MI$. Also,
\[\|h\|=\|T^{-1}Th\|\leq \|T^{-1}\|\|Th\|,\]
hence
\[\|T^{-1}\|^{-1}\|h\|\leq \|Th\|\]
for every $h\in \h$. Suppose that $m=\|T^{-1}\|^{-1}$. Since $T$ is
positive, $mI\leq T$.
\end{proof}
The following proposition shows that any $C$-controlled frame is a
controlled frame.
\begin{prop}\label{nessecery}
Let sequence $\Psi = \{\psi_j \in \h : j \in J\}$ be $C$-controlled
frame in Hilbert C*-module $\h$ for $C \in GL(\h)$. Then $\Psi$ is a
frame in Hilbert C*-module $\h$. Furthermore $CS = SC^*$ and so
\[\sum_{j\in J} \langle f, \psi_j \rangle C\psi_j=\sum_{j\in J} \langle f, C\psi_j \rangle \psi_j.\]
\end{prop}
\begin{proof}
Define $S:=C^{-1}S_C$. Then for every $f\in \h$
\[Sf= C^{-1}S_{C}f=C^{-1}\sum_{j\in J} \langle f, \psi_j \rangle C\psi_j =\sum_{j\in J} \langle f, \psi_j \rangle \psi_j.\]
The operator $S: \h \rightarrow \h$ is positive, invertible and
self-adjoint. Now by Theorem \ref{efo} and Theorem \ref{Balaz2010}
$\Psi$ is a frame.

Since the operator $S_C$ is self-adjoint and $S_C=CS$, then
$CS=S_C=S_C^*=S^*C^*=SC^*.$
\end{proof}
According to the following proposition, if the operator  $C$  is
self-adjoint, then $C$-controlled frames are equivalent to classical
frames. this is a generalization of Proposition 3.3 in \cite{BAG}
for Hilbert $C^*$-module setting.

\begin{prop}
Let $\h$ be a Hilbert $C^*$-module and $C \in GL(\h)$ be
self-adjoint. Then $ \{\psi_j \in \h : j \in J\}$ is a $C$-
controlled frame for $\h$ if and only if $\Psi$ is a  frame for $\h$
and $C$ is positive and commutes with frame operator $S$ .
\end{prop}
\begin{proof}
Let $ \{\psi_j : j \in J\}$ be a $C$-controlled frame for $\h$. Then
from Proposition \ref{nessecery} $ \{\psi_j  : j \in J\}$ is a frame
$\h$ and $C$ is  commutes with frame operator $S$. Therefore
$C=S_CC^{-1}$ is positive.

\par For the converse, we note that, if $ \{\psi_j : j \in J\}$ is a
frame, then the frame operator $S$ is positive and invertible.
Therefore $CS=S_C$ is positive and invertible. Now, by Theorem
\ref{Balaz2010} $ \{\psi_j : j \in J\}$ is a $C$-controlled frame
for $\h$.
\end{proof}

\section{Multipliers of controlled frames and weighted frames in Hilbert $C^*$-modules}

In this section, we generalize the concept of multipliers of frames
for controlled frames in Hilbert $C^*$-module. Then we investigate
weighted frames in Hilbert $C^*$-modules and verify their relation
to controlled frames and multiplier operators.

\begin{prop}
Let $\h$ be a Hilbert $C^*$-module and $C \in GL(\h)$. Assume $\Phi
= \{\phi_j \in \h : j \in J\}$ and $\Psi = \{\psi_j \in \h : j \in
J\}$ are $C$-controlled Bessel sequence for $\h$. Then the operator
\[M_{m, \Phi, \Psi}: \h\rightarrow \h\]
defined by
\[M_{m, \Phi, \Psi}f=\sum_{j\in J}m_j \langle f, \psi_j\rangle C\phi_j \]
is a well defined bounded operator.
\end{prop}
\begin{proof}
Let $\Phi = \{\phi_j \in \h : j \in J\}$ and $\Psi = \{\psi_j \in \h
: j \in J\}$ be $C$-controlled Bessel sequences for $\h$ with bounds
$D$ and $D'$, respectively. For any $f, g\in \h$  and finite subset
$I\subset J$,
\begin{align*}
   \left\|\sum_{i\in I} m_i\langle f, \psi_i\rangle C\phi_i \right\| & = sup_{g\in \h,
    \|g\|=1}\left \|\sum_{i\in I} m_i\langle f, \psi_i\rangle \langle C\phi_i, g\rangle
    \right\| \\ & \leq sup_{g\in \h,\|g\|=1} \left\|\left(\sum_{i\in I} |m_i|^2|\langle f, \psi_i\rangle|^2\right)^\frac{1}{2}
     \left(\sum_{i\in I} |\langle C\phi_i, g \rangle|^2\right)^\frac{1}{2}
    \right\| \\ & \leq sup_{g\in \h,\|g\|=1} \|m\|_\infty \left\|\left(\sum_{i\in I} |\langle f, \psi_i\rangle|^2\right)^\frac{1}{2}
     \left(\sum_{i\in I} |\langle C\phi_i, g \rangle|^2\right)^\frac{1}{2}
    \right\| \\ & \leq \|m\|_\infty \sqrt{DD'}\|f\|
\end{align*}
This show that $M_{m, \Phi, \Psi}$ is well defined and
\[\|M_{m, \Phi, \Psi}\|\leq \|m\|_{\infty}\sqrt{BB'}.\]
\end{proof}

Above lemma is a motivation to define the following definition.
\begin{defn}
Let $\h$ be a Hilbert $C^*$-module and $C \in GL(\h)$. Assume $\Psi
= \{\psi_j \in \h : j \in J\}$ and $\Phi = \{\phi_j \in \h : j \in
J\}$ are $C$-controlled Bessel sequences for $\h$.
 Then the operator
\[M_{m, \Phi, \Psi}: \h\rightarrow \h\]
defined by
\[M_{m, \Phi, \Psi}f=\sum_{j\in J}m_j \langle f, \psi_j\rangle C\phi_j\]
is called the $C$-controlled multiplier operator with symbol $m$.
\end{defn}

The following definition is a generalization of weighted frames in
Hilbert space to  Hilbert $C^*$-module.

\begin{defn}
Let $\Psi = \{\psi_j \in \h : j \in J\}$ be a sequence of elements
in Hilbert $C^*$-module $\h$ and $\{\omega_j\}_{j\in J} \subseteq Z(A)$
 a sequence of positive weights. This pair is called a $w$-frame of
Hilbert $C^*$-module $\h$ if there exist constants $C, D > 0$ and
such that
\[ C\langle f, f\rangle \leq \sum_{j\in J} \omega_j \langle f,
\psi_j \rangle \langle \psi_j, f \rangle \leq D\langle f,
f\rangle.\] for all $f\in \h$.
\end{defn}

A sequence$\{c_j: j\in J\}\in Z(A)$ is called semi-normalized if there are
bounds $b\geq a > 0$, such that $a\leq |c_n| \leq b$.

The following proposition gives a relation between controlled
frames, weighted frames and multiplier operators.

\begin{prop}\label{cwm}
Let $\h$ be a Hilbert $C^*$-module and $C \in GL(\h)$ be
self-adjoint and diagonal on $\Psi = \{\psi_j \in \h : j \in J\}$
and assume it generates a controlled frame. Then the sequence
$W=\{\omega_j\}_{j\in J} \subseteq Z(A) $, which verifies the relations
$C\psi_n = \omega_n\psi_n$, is semi-normalized and positive.
Furthermore $C = M_{W,\tilde{\Psi},\Psi}$.
\end{prop}
\begin{proof}
By Theorem \ref{Balaz2010}, we get the following result for
$C^{1/2}$: \[m \|f\|^2\leq \|C^{1/2}f\|^2\leq M \|f\|^2.\] As
$C\psi_j = \omega_j \psi_j$, clearly $C^{1/2}\psi_j =
\sqrt{\omega_j} \psi_j$. Applying the above inequalities  to the
elements of the sequence, we get $0< m\leq \omega_j \leq M$.

Clearly, for any $f\in \h$
\begin{align*}
Cf & = C\left( \sum_{j\in J} \langle f, \tilde{\psi}_j\rangle \psi_j
\right)= \sum_{j\in J} \langle f, \tilde{\psi}_j\rangle C\psi_j \\ &
=\sum_{j\in J} \langle f, \tilde{\psi}_j\rangle \omega_j\psi_j
=M_{W,\tilde{\Psi}, \Psi}f.
\end{align*}

\end{proof}

As a to the first part of Proposition \ref{cwm} a frame weighted by
semi-normalized sequence is always a frame. Indeed, we have the
following lemma.

\begin{lem}\label{semi}
Let $\{\omega_j: j\in J\}$ be a semi-normalized sequence with bounds
$a$ and $b$. If $\{\psi_j: j\in J\}$ is a frame with bounds $C$ and
$D$ in Hilbert $C*$-module $\h$, then $\{\omega_j\psi_j: j\in J \}$
is also a frame with bounds $a^2C$ and $b^2D$. The sequence
$\{\omega^{-1}_j \tilde{\psi}_j: j\in J\}$ is a dual frame of
$\{\omega_j\psi_j: j\in J\}$.
\end{lem}
\begin{proof}
Since for any $f\in \h$ $|\langle
f,\omega_j\psi_j\rangle|^2=|\omega_j|^2|\langle f,\psi_j\rangle|^2$,
we get
\[\Delta:=\sum_{j\in J} |\langle f,\omega_j\psi_j\rangle|^2=\sum_{j\in J} |\omega_j|^2|\langle
f,\psi_j\rangle|^2.\] Thus $\Delta\leq b^2\sum_{j\in J} |\langle
f,\omega_j\psi_j\rangle|^2\leq b^2D\|f\|^2$. In addition,
\[\Delta\geq a^2 \sum_{j\in J} \langle
f,\psi_j\rangle|^2 \geq a^2C\|f\|^2. \] As $\sum_{j\in J} \langle f,
\omega_j\psi_j\rangle \omega_j^{-1}\tilde{\psi_j}=\sum_{j\in
J}\langle f.\psi_j\rangle \tilde{\psi_j}=f$, these two sequences are
dual. Since $\omega_j^{-1}$ is bounded,
$\{\omega_j^{-1}\tilde{\psi_j}: j\in J \}$ is a Bessel sequence dual
to a frame. Therefore, it is a dual frame of $\{\omega_j\psi_j: j\in
J\}$.
\end{proof}

The following results give a connection between weighted frames and
frame multipliers.

\begin{lem}
  Let $\Psi = \{\psi_j \in \h : j \in J\}$ be a frame for Hilbert $C^*$-module $\h$. Let $m = \{m_j\}_{j\in J}$ be a positive and semi-normalized
sequence. Then the multiplier $M_{m,\Psi}$ is the frame operator of
the frame $\{\sqrt{m_j}\psi_j: j\in J\} $ and therefore it is
positive, self-adjoint and invertible. If $\{m_j\}_{j\in J}$ is
negative and semi-normalized, then $M_{m,\Psi}$ is negative,
self-adjoint and invertible.
\end{lem}

\begin{proof}
\[M_{m,\Psi}f=\sum_{j\in J} m_j\langle f, \psi_j\rangle \psi_j=\sum_{j\in J}\langle f, \sqrt{m_j}\psi_j\rangle\sqrt{m_j}\psi_j.\]
By Lemma \ref{semi}, $\{\sqrt{m_j}\psi_j: j\in J\}$ is a frame.
Therefore $M_{m,\Psi}=S_{(\sqrt{m_j}\psi_j)}$  is positive and
invertible.

Let $m_j<0$ for all $j\in J$, then $m_j=-\sqrt{|m_j|^2}$. Therefore
\[M_{m, \Psi}=-\sum_{j\in J}\langle f, \sqrt{m_j}\psi_j\rangle\sqrt{m_j}\psi_j=-S_{(\sqrt{m_j}\psi_j)}.\]
\end{proof}
\begin{thm}
  Let $\{\psi_j: j\in J\}$ be a sequence of elements in Hilbert $C^*$-module $\h$. Let $W=\{\omega_j : j\in J\} $ be a
sequence of positive and semi-normalized weights. Then the following
properties are equivalent:
\begin{enumerate}
  \item $\{\psi_j: j\in J\}$ is a frame;
  \item $M_{w,\Psi}$ is a positive and invertible operator;
  \item There are constants $C, D > 0$  such that for all $f\in \h$
\[ C\langle f, f\rangle \leq \sum_{j\in J} \omega_j \langle f,
\psi_j \rangle \langle \psi_j, f \rangle \leq D\langle f,
f\rangle.\] i.e. the pair $\{\omega_j: j\in J\}, \{\psi_j: j\in J\}$
forms a  weighted frame;
  \item $\{\sqrt{\omega_j}\psi_j: j\in J\}$ is a frame;
  \item $M_{w,\Psi}$ is a positive and invertible operator for any positive, semi-normalized
sequence $W' = \{\omega^{'}_j\}_{j\in J}$;
  \item  $(\omega_j\psi_j)$ is a frame, i.e. the pair $\{\omega_j: j\in J\}, \{\psi_j: j\in J\}$ forms a weighted frame.
\end{enumerate}

\end{thm}
\begin{proof}
It is similar to Theorem 4.5 of \cite{BAG}.
\end{proof}

\textbf{Acknowledgement:} The authors would like to thank the refree for useful and helpful comments and suggestions.



\begin{thebibliography}{99}
\bibitem{A} L. Arambaši´c, On frames for countably generated Hilbert $C^*$-modules, Proc. Amer. Math. Soc. 135,  469--478, (2007).
\bibitem{aris}
 M. L. Arias and M. Pacheco, {\em Bessel fusion multipliers}. Journal of Mathematical Analysis and Applications ,
348(2), 581-- 588, (2008).
\bibitem{peterthesis}
P. Balazs, {\it Regular and Irregular Gabor Multipliers with Application
to Psychoacoustic Masking}, Ph.D Thesis University of Vienna, (2005).
\bibitem{xxlmult1}
P.~Balazs,
\newblock {\it Basic definition and properties of Bessel multipliers},
\newblock {\em Journal of Mathematical Analysis and Applications},
  325(1), 571--585, (2007).

\bibitem{BAG}
P.~Balazs, J-P. Antoine and A. Grybos,  {\it Wighted and Controlled
Frames.} {\em International Journal of Wavelets, Multiresolution and
Information Processing},  8(1),  109--132, (2010).

  \bibitem{raba}
 P. Balazs, D. Bayer and A. Rahimi, {\it Multipliers for continuous
frames in Hilbert spaces,} J. Phy. A: Mathematical and Theoretical,
45,  (2012).

\bibitem{BVA} I. Bogdanova, P. Vandergheynst, J.P. Antoine, L. Jacques, M. Morvidone,
Stereographic wavelet frames on the sphere, Applied Comput. Harmon.
Anal. (19),  223--252, (2005).
\bibitem{C} {O. Christensen, {\it{}An Introduction to Frames and Riesz
Bases}, Birkhauser, Boston,} (2003).
\bibitem{Conway}J. B. Conway, A Course in Functional Analysis, 2nd edn. (Springer, New York, 1990).

\bibitem{DGM} {I. Daubechies, A. Grossmann,Y. Meyer, Painless nonorthogonal
expansions,{\it{} J. Math. Phys.} 27 ,} 1271--1283, (1986).
\bibitem{DS} {R. J. Duffin, A.C. Schaeffer, A class of nonharmonic Fourier
series, {\it{}Trans. Amer. Math. Soc.} 72, } 341-366, (1952).

\bibitem{FaOsRa}
M. H. Faroughi, E. Osgooei, A. Rahimi, {\em Some properties of
$(X_d, X^*_ d )$ and $(l^\infty, X_d, X^*_ d )$-Bessel multipliers},
Azarbijan Journal of Mathematics, 3(2), 70--78,  (2013).

\bibitem{RAS} M. H. Faroughi, E. Osgooei, A. Rahimi, \textit{$(X_{d}, X_{d}^{*})$-Bessel Multipliers in Banach spaces},
Banach J. Math. Anal., 7(2), 146--161, (2013).

\bibitem{FL1} {M. Frank, D.R. Larson, {\it A module frame concept for Hilbert
$C^*$-modules, in: Functional and Harmonic Analysis of Wavelets}, San
Antonio, TX, January 1999, {\it{}Contemp. Math.} 247, Amer. Math.
Soc., Providence, RI} 207--233, (2000).
\bibitem{FL} {M. Frank, D.R. Larson, {\it Frames in Hilbert $C^*$-modules
and $C^*$-algebras}, {\it{}J. Operator Theory} 48},  273--314,
(2002).
\bibitem{HJLM} {D. Han, W. Jing, D. Larson, R. Mohapatra, {\it Riesz bases and their dual modular frames in Hilbert
$C^*$-modules}, {\it{}J. Math. Anal. Appl.} 343},  246--256, (2008).

\bibitem{HJM} {D. Han, W. Jing, R. Mohapatra, {\it  Perturbation of frames and Riesz bases in Hilbert $C^*$-modules},
{\it{}Linear Algebra Appl.} 431}, 746--759, (2009).

\bibitem{jing} W. Jing, {\it Frames in Hilbert $C^*$-modules}, {\it{}Ph. D. thesis, University of Central
Florida Orlando, Florida} (2006).

\bibitem{kasp}
G. G. Kasparov, {\it Hilbert $C^*$-modules: theorems of Stinespring
and Voiculescu,} J. Operator Theory,  4, 133--150, (1980).

\bibitem{kap}
I. Kaplansky, {\it Algebra of type I,} Ann. Math., 56, 460--472,
(1952).

\bibitem{KK1} {A. Khosravi, B. Khosravi, {\it Frames and bases in tensor products of Hilbert spaces and Hilbert
$C^*$-modules}, {\it{}Proc. Indian Acad. Sci. Math. Sci.} 117},
1--12, (2007).
\bibitem{KK2} A. Khosravi, B. Khosravi, {\it g-frames and modular Riesz
bases in Hilbert $C^*$-modules}, {\it{}Int. J. Wavelets
Multiresolut. Inf. Process.}, 10(2), 1--12, (2012).
\bibitem{KM} A. Khosravi, M. Mirzaee Azandaryani, {\it Bessel multipliers in Hilbert $C^*$-modules}, Banach J. Math. Anal., 9(3), 153--163, (2015).  
\bibitem{L} E. C. Lance, {\it Hilbert $C^*$-Modules: A Toolkit for Operator Algebraists}, London Math. Soc. Lecture Note
Ser. 210, Cambridge Univ. Press, (1995).

\bibitem{5}
A. Rahimi, {\it Multipliers of Genralized frames in Hilbert
spaces,} Bulletin of the Iranian Mathematical Society, Vol., 37(1),
63--80,  (2011).

\bibitem{AsPe}
A. Rahimi and P. Balazs, {\it Multipliers of p-Bessel sequences in
Banach spaces,} Integral Equations and Operator Theory,  68(2),
193--205, (2010).
\bibitem{rafe}
A. Rahimi and A. Fereydooni, {\it Controlled G-Frames and Their
G-Multipliers in Hilbert spaces}
 An. St. Univ. Ovidius Constanta, 21(2),  223--236, (2013).
\bibitem{RN1} M. Rashidi-Kouchi, A. Nazari, {\it Equivalent continuous g-frames in Hilbert $C^*$-modules,}
 Bull. Math. Anal. Appl.,  4 (4), 91--98, (2012).
\bibitem{RN} M. Rashidi-Kouchi, A. Nazari, {\it Continuous g-frame in Hilbert $C^*$-modules,} Abst. Appl. Anal., 2011, 1--20, (2011).
\bibitem{RNM} M. Rashidi-Kouchi, A. Nazari, M. Amini, {\it On stability of g-frames and g-Riesz bases in Hilbert $C^*$-modules,}
 Int. J.Wavelets Multiresolut. Inf. Process, 12(6), 1--16, (2014).
\bibitem{schatt1}
R.~Schatten,
\newblock {\em Norm Ideals of Completely Continious Operators}.
\newblock Springer Berlin, (1960).

\bibitem{wanbro06}
D. Wang and Guy~J. Brown,
\newblock {\em Computational Auditory Scene Analysis: Principles, Algorithms,
  and Applications}.
\newblock Wiley-IEEE Press, (2006).
 \bibitem{O} N. E. Wegge-Olsen, {\it
K-theory and C*-algebras  a Friendly Approach}, Oxford University
Press, Oxford, England, (1993).


\end{thebibliography}
\end{document}